\documentclass[11pt,reqno]{amsart}

\usepackage{a4wide}
\usepackage{dsfont} 
\usepackage{amssymb,amsmath}
\usepackage{mathrsfs}
\usepackage{comment}
\usepackage[all]{xy}

\newtheorem{proposition}{Proposition}[section]
  \newtheorem{theorem}[proposition]{Theorem}
  \newtheorem{corollary}[proposition]{Corollary}
  \newtheorem{lemma}[proposition]{Lemma}
\theoremstyle{remark}
  \newtheorem{definition}[proposition]{Definition}

\newcommand{\cst}{\ifmmode\mathrm{C}^*\else{$\mathrm{C}^*$}\fi}
\newcommand{\st}{\;\vline\;}
\newcommand{\CC}{\mathbb{C}}
\newcommand{\NN}{\mathbb{N}}
\newcommand{\tens}{\otimes}
\newcommand{\atens}{\otimes_{\text{\tiny{alg}}}} 
\newcommand{\id}{\mathrm{id}}
\newcommand{\Pol}{\mathrm{Pol}}
\newcommand{\comp}{\!\circ\!}
\newcommand{\I}{\mathds{1}}
\newcommand{\vt}{\!\vartriangle\!}
\newcommand{\hh}[1]{\widehat{#1}}
\newcommand{\GG}{\mathbb{G}}
\newcommand{\QG}{\mathbb{G}}
\newcommand{\HH}{\mathbb{H}}
\newcommand{\KK}{\mathbb{K}}
\newcommand{\balpha}{\boldsymbol{\alpha}}
\newcommand{\cA}{\mathscr{A}}
\newcommand{\cB}{\mathscr{B}}
\newcommand{\cF}{\mathscr{F}}
\newcommand{\sB}{\mathsf{B}}
\newcommand{\sC}{\mathsf{C}}
\newcommand{\sS}{\mathsf{S}}
\newcommand{\bbeta}{\boldsymbol{\beta}}
\newcommand{\opp}{\text{\tiny{op}}}
\renewcommand{\Bar}[1]{\overline{#1}}

\DeclareMathOperator{\C}{C}
\DeclareMathOperator{\Fun}{Fun}
\DeclareMathOperator{\ord}{ord}
\DeclareMathOperator{\QAUT}{QAUT}
\DeclareMathOperator{\qAut}{qAut}

\numberwithin{equation}{section}

\begin{document}

\author{Jyotishman Bhowmick}
\address{Stat-Math Unit, Indian Statistical Institute, 203, B. T. Road, Kolkata 700 208} \email{jyotishmanb@gmail.com}

\author{Adam Skalski}
\address{Institute of Mathematics of the Polish Academy of Sciences,
ul.~\'Sniadeckich 8, 00--956 Warszawa, Poland
\newline \indent Faculty of Mathematics, Informatics and Mechanics, University of Warsaw, ul.~Banacha 2,
02-097 Warsaw, Poland}
\email{a.skalski@impan.pl}

\author{Piotr M.~So{\l}tan} \address{Department of Mathematical Methods in Physics, Faculty of Physics, University of Warsaw, Poland}
\email{piotr.soltan@fuw.edu.pl}

\thanks{AS and PS were partially supported by National Science Centre (NCN) grant no.~2011/01/B/ST1/05011.   JB wishes to thank Sergey Neshveyev and the Department of Mathematics, University of Oslo, where he was a post-doctoral fellow when this work started}

\title[Quantum automorphisms of finite quantum groups]{Quantum group of automorphisms of a finite quantum group}


\begin{abstract}
A notion of a quantum automorphism group of a finite quantum group, generalising that of a classical automorphism group of a finite group, is proposed and a corresponding existence result proved.
\end{abstract}

\maketitle

The story of quantum symmetry groups began with the paper \cite{Wang}, where S.\,Wang defined and began to study quantum symmetry groups of finite-dimensional \cst-algebras. Soon after that T.\,Banica, J.\,Bichon and others expanded this study to quantum symmetry groups of various finite structures, such as (coloured) graphs or metric spaces (for the information on infinite-dimensional extensions, among them the quantum isometry groups of D.\,Goswami, we refer to \cite{orth} and references therein). The general idea behind these concepts is based on considering all compact quantum group actions on a given finite quantum space -- viewed dually as a finite-dimensional \cst-algebra -- preserving some extra structure of that space, and looking for a final object in the resulting category.

In this short note we propose studying in this spirit the quantum group of all quantum automorphisms of a given finite quantum or classical group. The starting point of our approach is based on an observation saying that a transformation of a finite abelian group $\Gamma$ is an automorphism if and only if it induces in a natural way, via the Fourier transform, a transformation of $\hh{\Gamma}$, the Pontriagin dual of $\Gamma$. In particular the automorphism groups of $\Gamma$ and $\hh{\Gamma}$ are canonically isomorphic.
We thus define a quantum family of maps on a finite quantum group $\GG$ to be a \emph{quantum family of automorphisms} if it induces, via the Fourier transform associated to $\GG$, a quantum family of maps on the dual quantum group $\hh{\GG}$. We show that  a universal quantum family of automorphisms of $\GG$ exists, and naturally defines the \emph{quantum automorphism group of $\GG$}, which is a compact quantum group in the sense of Woronowicz. Its classical version is the group of all automorphisms of $\GG$, i.e.\ these automorphisms of the \cst-algebra $\C(\GG)$ which commute in a natural sense with the coproduct of the latter algebra.

It has to be observed that we do not know any example in which the quantum automorphism group as defined in this note is not a classical group. Thus the notion, though apparently natural and satisfactory, needs to be treated as tentative and open to modifications. In particular the main question of interest, i.e.\ the problem which classical finite groups admit genuinely quantum automorphisms, remains open.

The plan of the article is as follows: in Section 1 we recall the properties of the duality and Fourier transform on finite quantum groups, as defined and studied for example by A.\,Van Daele (\cite{VDduality}, \cite{VDFourier}). In Section 2 we introduce the notion of a quantum family of automorphisms of a given finite quantum group $\GG$, discuss some equivalent conditions related to that definition and use it  to prove the existence of the quantum automorphism group of $\GG$, which is canonically isomorphic to that of  $\hh{\GG}$. In section 3 we discuss some simplifications and special properties appearing when one considers quantum automorphism groups of classical groups. The symbol $\tens$ will denote both the spatial tensor product of $C^*$-algebras and the algebraic tensor product of vector spaces (if we want to stress we are using the latter we write $\tens_{\textup{alg}}$).

{\bf Acknowledgement.} We would like to thank the anonymous referee for an exceptionally careful reading of our article and thoughtful comments that led to substantial improvements of its contents. In particular Theorem \ref{cyclic} and its proof are due to the referee.

\section{Finite quantum groups, duality and the Fourier transform}\label{FourierSection}

Let $\GG$ be a finite quantum group. In the following sections we will denote the finite dimensional \cst-algebra corresponding to $\GG$ by the symbol $\C(\GG)$. However, in order to keep the notation lighter, throughout this section the quantum group $\GG$ will be fixed and we will write $\cA$ for the algebra $\C(\GG)$. Then $\cA$ is a finite dimensional Hopf $*$-algebra whose coproduct, antipode and counit will be denoted by $\Delta$, $S$ and $\epsilon$ respectively. We will use the symbol $h$ to denote the Haar state of $\GG$ and occasionally employ the \emph{Sweedler notation} for the coproduct: $\Delta(a):=a_{(1)} \tens a_{(2)}$, $a \in \cA$. The \emph{convolution product} of two elements $a,b\in\cA$ is given by
\begin{equation}\label{conv}
a\star{b}=(h\tens\id)\Bigl(\bigl((S\tens\id)\Delta(b)\bigr)(a\tens\I)\Bigr).
\end{equation}
As $\GG$ is finite, the antipode $S$ is an involution commuting with the usual adjoint of $\cA$ and $h$ is a trace (we will use these facts in what follows without further comment). Hence \eqref{conv} coincides with the formula proposed in \cite[Proposition 2.2]{VDFourier}, i.e.~$a\star{b}=h\bigl(S^{-1}(b_{(1)})a\bigr)b_{(2)}$. Note that a different formula for the convolution product is used in \cite{PodW}. The convolution product is in fact an associative bilinear operation making $\cA$ into an involutive algebra with involution $\bullet$ defined as
\begin{equation}\label{convadj}
\cA\ni{a}\longmapsto{a^\bullet}=S(a^*)\in\cA.
\end{equation}
This involution is referred to as the \emph{convolution adjoint}.

Let us note here the important relation between the Haar state and antipode: for any $b,c\in\cA$ we have
\begin{equation}\label{HW}
S\Bigl((\id\tens{h})\bigl(\Delta(b)(\I\tens{c})\bigr)\Bigr)
=(\id\tens{h})\bigl((\I\tens{b})\Delta(c)\bigr).
\end{equation}
This relation can be found in \cite[Proof of Proposition 3.11]{VDduality} (see also \cite[Lemma 5.5]{mha}). It is worth mentioning that in the context of Kac algebras \eqref{HW} is taken as the defining property of $h$ (\cite[Section 2.2]{EnockSchwartz}).

Every finite quantum group $\GG$ possesses a (unique) element  $\eta\in\cA$ such that $\epsilon(\eta)=1$ and
\[
a\eta=\epsilon(a)\eta,\qquad{a}\in\cA.
\]
In fact every finite-dimensional Hopf algebra $\cB$ admits an element satisfying the properties listed above, often called a \emph{Haar element in} $\cB$ (see \cite{LarSweed}, Chapter 5 in \cite{Sweed} and \cite{Haarfinite} for the quantum group case).
For example, when $\GG=G$ for a finite group $G$, we have $\eta=\delta_{e}$ and $h(\eta)=\frac{1}{|G|}$. On the other hand if $\GG=\hh{G}$ for a finite group $G$, i.e.~$\cA=\CC[G]$, then $\eta=\frac{1}{|G|}\sum_{g\in\Gamma}\lambda_g$, but we still have $h(\eta)=\frac{1}{|G|}$ (see also Subsection \ref{FundEx}). More generally one can show that $h(\eta)=\tfrac{1}{\dim\cA}$ (cf.~\cite[Section A.2]{pseudogr}).

Let us now briefly describe the \emph{dual} $\hh{\GG}$ of the quantum group $\GG$. The standard notation for \cst-algebra corresponding to $\hh{\GG}$ is $\C(\hh{\GG})$, but in this section we will use $\hh{\cA}$ to denote it. As a vector space $\hh{\cA}$ is defined to be the set
\[
\bigl\{h(\cdot\,a)\st{a}\in\cA\bigr\}.
\]
Clearly it is a subspace of the dual space of $\cA$, but thanks to faithfulness of the Haar state $h$, it is in fact the whole of $\cA^*$. The isomorphism of vector spaces
\[
\cA\ni{a}\longmapsto{h(\cdot\,a)}\in\hh{\cA}
\]
is called the \emph{Fourier transform} and is denoted by $\cF$. The space $\cA$ can be equipped with the structure of a Hopf $*$-algebra which we will describe below.

\subsection{Unital $*$-algebra structure} The product in $\hh{\cA}$ is defined as convolution of functionals: for $\omega_1,\omega_2\in\hh{\cA}$ we define $\omega_1\omega_2=(\omega_1\tens\omega_2)\comp\Delta$. One easily checks that this is indeed an associative multiplication on $\hh{\cA}$ and the counit of $\cA$ is the unit of $\hh{\cA}$. We will often write $\hh{\I}$ to denote the unit of $\hh{\cA}$. The involution making $\cA$ a $*$-algebra is the mapping $\omega\mapsto\omega^*$, where
\[
\omega^*(a)=\Bar{\omega\bigl(S(a)^*\bigr)},\qquad{a}\in\cA.
\]

In terms of the Fourier transform the structure described above has the following property (which can be taken as a definition of $\star$):
\begin{equation} \label{Fourierconv}
\cF(a\star{b})=\cF(a)\cF(b),\qquad{a},b\in\cA.
\end{equation}
To see this let us take any $c\in\cA$. Then
\[
\begin{split}
\cF(a\star{b})(c)
&=h\bigl((a\star{b})c\bigr)
 =h\Bigl((h\tens\id)\bigl(((S\tens\id)\Delta(b))(a\tens\I)\bigr)c\Bigr)\\
&=(h\tens{h})\Bigl(\bigl((S\tens\id)\Delta(b)\bigr)(a\tens{c})\Bigr)
 =h\Bigl((\id\tens{h})\bigl(((S\tens\id)\Delta(b))(\I\tens{c})\bigr)a\Bigr)\\
&=h\Bigl((\id\tens{h})\bigl((\I\tens{b})\Delta(c)\bigr)a\Bigr)
 =(h\tens{h})(c_{(1)}a\tens{b}c_{(2)})\\
&=h(ac_{(1)})h(bc_{(2)})
 =\bigl(\cF(a)\tens\cF(b)\bigr)\Delta(c)
=\bigl(\cF(a)\cF(b)\bigr)(c),
\end{split}
\]
where we used \eqref{HW} and traciality of $h$.

The involution of $\hh{\cA}$ is also easily expressed with help of the Fourier transform. Indeed, using the fact that $h$ is positive and $S$-invariant, we obtain
\[
\begin{aligned}
\bigl(\cF(b)\bigr)^*(a)&=\Bar{\cF(b)\bigl(S(a)^*\bigr)}\\
&=\Bar{h\bigl(S(a)^*b\bigr)}\\
&=h\bigl(b^*S(a)\bigr)=h\bigl(aS(b^*)\bigr)\\
&=\cF\bigl(S(b^*)\bigr)(a)
\end{aligned}
\]
For all $a,b\in\cA$. In other words
\begin{equation}\label{Fbu}
\cF(b)^*=\cF\bigl(S(b^*)\bigr)=\cF(b^\bullet),\qquad{b}\in\cA.
\end{equation}

Finally, the unit $\hh{\I}$ of $\hh{\cA}$ can be written $\cF\bigl(\tfrac{1}{h(\eta)}\eta\bigr)$, as follows immediately from the properties of $\eta$.

\subsection{The coproduct}\label{hhDel}
We identify (canonically) $\hh{\cA}\tens\hh{\cA}$ with $(\cA\tens\cA)^*$. Now, given $\omega\in\hh{\cA}$ we define $\hh{\Delta}(\omega)$ as the element of $\hh{\cA}\tens\hh{\cA}$ for which we have
\begin{equation}\label{Delhat}
\bigl(\hh{\Delta}(\omega)\bigr)(c_1\tens{c_2})=\omega(c_1c_2),\qquad{c_1},c_2\in\cA.
\end{equation}
This defines a coassociative coproduct $\hh{\Delta}\colon\hh{\cA}\to\hh{\cA}\tens\hh{\cA}$.

In terms of the Fourier transform we have $\hh{\Delta}\bigl(\cF(a)\bigr)=\sum\limits_{i}\cF(a_i)\tens\cF(b_i)$ (for $a \in \cA$) if and only if
\[
\sum_ih(c_1a_i)h(c_2b_i)=h(c_1c_2a)
\]
for all $c_1,c_2\in\cA$.

\subsection{The counit, the antipode and the Haar measure} The counit $\hh{\epsilon}$ of $\hh{\cA}$ is the functional given by evaluating $\omega\in\hh{\cA}$ in the unit $\I$ of $\cA$, as follows from \eqref{Delhat}. In terms of the Fourier transform we have
\[
\hh{\epsilon}\bigl(\cF(a)\bigr)=h(a),\qquad{a}\in\cA.
\]

The antipode is defined directly via the duality of vector spaces, i.e.~$\hh{S}(\omega)=\omega\comp{S}$ for $\omega\in\cA^*\cong\hh{\cA}$. Using the fact that $h$ is $S$-invariant and a trace we can write $\hh{S}$ in terms of the Fourier transform as follows ($a,b\in \cA$):
\[
\hh{S}\bigl(\cF(b)\bigr)(a)=\cF(b)\bigl(S(a)\bigr)=h\bigl(S(a)b\bigr)=h\bigl(S(b)a\bigr)=h\bigl(aS(b)\bigr)=\cF\bigl(S(b)\bigr)(a),
\]
so that
\begin{equation}\label{SFFS}
\hh{S}\comp\cF={\cF}\comp{S}.
\end{equation}

The Haar measure of $\hh{\cA}$ is defined as
\[
\hh{h}\bigl(\cF(a)\bigr)=h(\eta)\epsilon(a),\qquad{a}\in\cA.
\]
This definition is different from the one given by Van Daele because he did not require $\hh{h}$ to be a state (the original formula did not incorporate the constant $h(\eta)$). 

\subsection{The Haar element, the convolution product and the convolution adjoint}

It can be easily checked that the Haar element of the dual quantum group, $\hh{\eta}$, is equal to $\cF(\I)$. Indeed, for any $a\in\cA$ we have $\I\star{a}=a\star\I=h(a)\I$, so
\[
\cF(\I)\cF(a)=\cF(\I\star{a})=h(a)\cF(\I)=\hh{\epsilon}\bigl(\cF(a)\bigr)\cF(\I)
\]
and the other required equality follows similarly. Furthermore, by \eqref{SFFS} and \eqref{Fbu} we have
\[
\cF(a)^\bullet=\hh{S}\bigl(\cF(a)\bigr)^*=\cF\bigl(S(a)\bigr)^*=\cF\bigl(S(a)^\bullet\bigr)=\cF(a^*),\qquad{a}\in\cA.
\]

Note also another useful formula which holds for all $a,b\in\cA$:
\[
\epsilon(a\star{b})
=\epsilon\Bigl((h\tens\id)\bigl((S\tens\id)\Delta(b)(a\tens\I)\bigr)\Bigr)
=h\Bigl((\id\tens\epsilon)\bigl((S\tens\id)\Delta(b)(a\tens\I)\bigr)\Bigr)
=h\bigl(S(b)a\bigr).
\]

Now let us examine the convolution product on $\hh{\cA}$ in terms of the Fourier transform. We have
\begin{equation} \label{Fourierconvdual}
h(\eta)\,\cF(ab)=\cF(b)\star\cF(a),\qquad{a},b\in\cA.
\end{equation}
Indeed, write $\hh{\Delta}\bigl(\cF(a)\bigr)=\sum\limits_{i}\cF(c_i)\tens\cF(d_i)$ and compute
\[
\begin{split}
\cF(b)\star\cF(a)
&=(\hh{h}\tens\id)\Bigl(\bigl((\hh{S}\tens\id)\hh{\Delta}(\cF(a)\bigr)(\cF(b)\tens\I)\Bigr)\\
&=\sum\limits_{i}(\hh{h}\tens\id)\Bigl(\bigl(\hh{S}(\cF(c_i))\tens\cF(d_i)\bigr)(\cF(b)\tens\I)\Bigr)\\
&=\sum\limits_{i}(\hh{h}\tens\id)\bigl(\cF(S(c_i))\cF(b)\tens\cF(d_i)\bigr)
 =\sum\limits_{i}\hh{h}\bigl(\cF(S(c_i)\star{b})\bigr)\cF(d_i)\\
&=h(\eta)\sum\limits_{i}\epsilon\bigl(S(c_i)\star{b}\bigr)\cF(d_i)
 =h(\eta)\sum\limits_ih\bigl(S(b)S(c_i)\bigr)\cF(d_i)\\
&=h(\eta)\sum\limits_ih(c_ib)\cF(d_i).
\end{split}
\]
It remains to see that $\sum\limits_ih(c_ib)\cF(d_i)=\cF(ab)$. To that end note that for any $x\in\cA$
\[
\sum_ih(c_ib)\cF(d_i)(x)=\sum_ih(c_ib)h(d_ix)=\sum_ih(bc_i)h(xd_i)=h(bxa)=h(abx)=\cF(ab)(x),
\]
where the third equality is a consequence of the definition of $\hh{\Delta}$ (cf.~Section \ref{hhDel}). This proves formula \eqref{Fourierconvdual}.

The reason for the extra flip and the normalizing factor can be seen from Lemma \ref{Fouriter} below (if one remembers that $S$ is anti-homomorphic).

\subsection{Iteration of the Fourier transform}

\begin{lemma} \label{Fouriter}
We have $\hh{\cF}\comp\cF=h(\eta)S$. In particular, the Fourier transform, when suitably rescaled, is a ``transformation of order 4''.
\end{lemma}

\begin{proof}
 The proof is an explicit calculation. Let $a\in\cA$. Put $b=\hh{\cF}\bigl(\cF(a)\bigr)\in\bigl((\cA)^*\bigr)^*$. We will show that for any $\omega\in\cA^*$ we have $\omega\bigl(h(\eta) S(a)\bigr)=b(\omega)$. We can assume that $\omega=\cF(c)$ for an element $c\in\cA$. Then
\[
\begin{split}
b(\omega)&=\hh{h}\bigl(\omega\cF(a)\bigr)\\
&=\hh{h}\bigl(\cF(c)\cF(a)\bigr)\\
&=\hh{h}\bigl(\cF(c\star{a})\bigr)\\
&=h(\eta)\epsilon(c\star{a})\\
&=h(\eta)(h\tens\epsilon)\bigl((S\tens\id)\Delta(a)(c\tens\I)\bigr)\\
&=h(\eta)h\bigl(S(\id\tens\epsilon)(\Delta(a))c\bigr)\\
&=h(\eta)h\bigl(S(a)c\bigr)\\
&=h(\eta)\cF(c)\bigl(S(a)\bigr)=\omega\bigl(h(\eta)S(a)\bigr).
\end{split}
\]
This ends the proof.
\end{proof}

\subsection{Fundamental examples}\label{FundEx}

\subsubsection{Algebra of functions on a finite group}

Let $G$ be a finite group and let $\cA=\Fun(G)$ with the standard pointwise structure. The coproduct can be written in the basis $\{\delta_g\}_{g\in{G}}$ as
\[
\Delta(\delta_g)=\sum_{ab=g}\delta_a\tens\delta_b,\qquad{g}\in G.
\]
Note that further we have
\[
\delta_g^{\bullet}=\delta_{g^{-1}},\quad\text{and}\quad\delta_g\star\delta_h=\frac{1}{|G|}\delta_{gh},\qquad{g},h\in{G}.
\]

The product $\cF(\delta_{g_1})\cF(\delta_{g_2})$ of elements of $\hh{\cA}$ is the functional
\[
\cA\ni{a}\longmapsto\bigl(h(\cdot\,\delta_{g_1})\tens{h(\cdot\,\delta_{g_2})\bigr)}\Delta(a).
\]
For $a=\delta_g$ we find the relevant value to be
\[
(h\tens{h})\sum_{ab=g}\delta_a\delta_{g_1}\tens\delta_b\delta_{g_2}
=\begin{cases}
\tfrac{1}{|G|^2},&g_1g_2=g,\\
0,&\text{else}.
\end{cases}
\]
This means that for all $g_1,g_2\in{G}$
\[
\cF(\delta_{g_1})\cF(\delta_{g_2})=\tfrac{1}{|G|}\cF(\delta_{g_1g_2}).
\]
Similarly, using the description of $\hh{\Delta}$ given in the Subsection \ref{hhDel}, we find that for all $g\in{G}$
\[
\hh{\Delta}\bigl(\cF(\delta_g)\bigr)=\tfrac{1}{|G|}\cF(\delta_g)\tens\cF(\delta_g).
\]
Thus the map
\[
\hh{\cA}\ni\cF(\delta_g)\longmapsto\tfrac{1}{|G|}\lambda_g\in\CC[G]
\]
is an isomorphism of Hopf algebras. Note that $\tfrac{1}{|G|}=h(\eta)$.

\subsubsection{Group ring of a finite group}

Let $G$ be a finite group and let $\cA$ denote this time the group ring of $G$, $\CC[G]$: we view it now naturally as the algebra of functions on the dual quantum group $\hh{G}$. The coproduct on $\cA$ is given by the formula
\[
\Delta(\lambda_g)=\lambda_g\tens\lambda_g,\qquad{g}\in{G},
\]
and further
\[
\lambda_g^{\bullet}=\lambda_g,\quad\text{and}\quad\lambda_g\star\lambda_h=\delta_{g,h}\lambda_h,\qquad{g},h\in{G}.
\]

This time the Fourier transform $\cF$ maps $\CC[G]$ onto $\Fun(G)$, and the calculations similar to these above (or an application of Lemma \ref{Fouriter}) yield in this picture the following formula:
\[
\cF(\lambda_g)=\delta_{g^{-1}},\qquad{g}\in{G}.
\]

\section{Quantum automorphisms of a finite quantum group}\label{Secdefautomor}



Throughout this section $\GG$ will denote a finite quantum group. The corresponding finite dimensional \cst-algebra playing the role of the algebra of functions on $\GG$ will be denoted by the symbol $\C(\GG)$. If $\sB$ is a unital \cst-algebra and $\alpha\colon\C(\GG)\to\C(\GG)\tens\sB$ a linear map, then following \cite{qs}, we say that $\alpha$ \emph{represents a quantum family of  maps on} $\GG$ if it is a unital $*$-homomorphism. We say that it \emph{represents a quantum family of invertible maps on} $\GG$ if in addition the \emph{Podle\'s condition} holds: $\alpha\bigl(\C(\GG)\bigr)(\I\tens\sB)$ spans $\C(\GG)\tens\sB$.

Due to finite dimensionality of $\C(\GG)$, the last condition is purely vector space theoretic: it means that the set of elements of the form $\bigl\{(\id\tens\omega)\bigl(\alpha(a)\bigr)\st{a}\in\C(\GG),\;\omega\in\sB^*\bigr\}$ spans $\C(\GG)$.

Given a linear map $\alpha\colon\C(\GG)\to\C(\GG)\tens\sB$  we define another linear map $\hh{\alpha}\colon\C(\hh{\GG})\to\C(\hh{\GG})\tens\sB$ by the formula
\begin{equation}\label{hatmap}
\hh{\alpha}=\frac{1}{h(\eta)}(\cF\tens\id_{\sB})\comp\alpha\comp\hh{\cF}\comp\hh{S}.
\end{equation}
Note that $\hh{\cF}$ denotes here (as in the previous section) simply the Fourier transform associated to the dual quantum group $\hh{\GG}$.

It may be noteworthy that the formula for $\hh{\alpha}$ can be expressed by another formula which involves the Fourier transform on $\GG$ alone. More precisely, we have:

\begin{proposition}\label{alhatInne}
$\hh{\alpha}=(\cF\tens\id)\comp\alpha\comp\cF^{-1}$.
\end{proposition}

\begin{proof}
Using Lemma \ref{Fouriter}, we have $\hh{\cF}\comp\hh{S}^{-1}=\hh{h}(\hh{\eta}){\cF}^{-1}$, i.e.~$\hh{\cF}\comp\hh{S}=\hh{h}(\hh{\eta}){\cF}^{-1}.$ Thus, $\hh{\alpha}=\frac{1}{h(\eta)}(\cF\tens\id_{\sB})\comp\alpha\comp\hh{\cF}\comp\hh{S}
=\frac{\hh{h}(\hh{\eta})}{h(\eta)}(\cF\tens\id)\comp\alpha\comp\cF^{-1}=(\cF\tens\id)\comp\alpha\comp\cF^{-1}$, since $\hh{h}(\hh{\eta})=h(\eta)$.
\end{proof}

Before we formulate the next lemma we need another piece of terminology: we say that $\alpha$ as above \emph{preserves the convolution product} if $\alpha$ is a homomorphism from $\bigl(\C(\GG),\star\bigr)$ (i.e.~$\C(\GG)$ equipped with the convolution product) to $\C(\GG)\tens\sB$ with convolution product on the first factor and the given product of $\sB$ on the second (so if $\mu\colon\C(\GG)\tens\C(\GG)\to\C(\GG)$ denotes the convolution product and $m\colon\sB\atens\sB\to\sB$ the product on $\sB$ then the condition on $\alpha$ is that it is a homomorphism from $\C(\GG)$ with convolution product to $\C(\GG)\tens\sB$ with the product $(\mu\tens{m})(\id_{\C(\GG)}\tens\sigma\tens\id_\sB)$, where $\sigma$ is the flip $\sB\tens\C(\GG)\to\C(\GG)\tens\sB$).

\begin{lemma}\label{dualactionalg}
Let $\sB$ be a unital \cst-algebra and $\alpha\colon\C(\GG)\to\C(\GG)\tens\sB$ a linear map. Then
\begin{enumerate}
\item\label{dualactionalg1} $\hh{\alpha}$ is a homomorphism from $\C(\hh{\GG})$ to $\C(\hh{\GG})\tens\sB$ (with usual product) if and only if $\alpha$ preserves the convolution product,
\item\label{dualactionalg2} $\hh{\alpha}$ is $*$-preserving if and only if $\alpha$ preserves the convolution adjoint (i.e.~$(\bullet\tens*)\comp\alpha=\alpha\comp\bullet$),
\item\label{dualactionalg3} $\hh{\alpha}$ is unital if and only if $\alpha$ preserves the Haar element (i.e.~$\alpha(\eta)=\eta\tens\I_{\sB}$),
\item\label{dualactionalg4} $\hh{\alpha}$ preserves the Haar state if and only if $\alpha$ preserves the counit.
\end{enumerate}

Moreover we have the following equality:
\begin{equation}\label{displayed}
\hh{\;\!\!\hh{\alpha}}=(S\tens\id_{\sB})\comp\alpha\comp{S},
\end{equation}
so that if $\beta=\hh{\;\!\!\hh{\alpha}}$, then $\alpha=\hh{\!\hh{\beta}}$.
\end{lemma}

\begin{proof}
All the statements in the lemma follow from the properties of the Fourier transform established in Section \ref{FourierSection}.

Proof of \eqref{dualactionalg1}: Let us denote the product maps on $\C(\GG)$ and $\C(\hh{\GG})$ by $m$ and $\hh{m}$ respectively. Then let $m_\sB$ be the product of $\sB$ and $\mu$ and $\hh{\mu}$ the convolution products on $\C(\GG)$ and $\C(\hh{\GG})$ respectively. We will use the symbol $\sigma$ to denote flip maps on various tensor products. We have
\begin{equation}\label{to1}
\begin{split}
(\hh{m}\tens{m_\sB})&\comp(\id\tens\sigma\tens\id)\comp(\hh{\alpha}\tens\hh{\alpha})\\
&=\frac{1}{h(\eta)^2}(\hh{m}\tens{m_\sB})(\id\tens\sigma\tens\id)\comp
(\hh{\cF}\tens\id\tens\hh{\cF}\tens\id)\comp
(\alpha\tens\alpha)\comp
(\hh{\cF}\tens\hh{\cF})\comp
(\hh{S}\tens\hh{S})\\
&=\frac{1}{h(\eta)^2}(\cF\tens\id)\comp(\mu\tens{m_\sB})(\id\tens\sigma\tens\id)\comp
(\alpha\tens\alpha)\comp
(\hh{\cF}\tens\hh{\cF})\comp
(\hh{S}\tens\hh{S})\\
\end{split}
\end{equation}
On the other hand using antimultiplicativity of $\hh{S}$, a dual version of formula \eqref{Fourierconvdual} and the fact that $\hh{h}(\hh{\eta})=h(\eta)$ we find that
\begin{equation}\label{to2}
\begin{split}
\hh{\alpha}\comp\hh{m}&=\frac{1}{h(\eta)}(\cF\tens\id)\comp\alpha\comp\hh{\cF}\comp\hh{S}\comp\hh{m}\\
&=\frac{1}{h(\eta)}(\cF\tens\id)\comp\alpha\comp\hh{\cF}\comp\hh{m}\comp\sigma\comp(\hh{S}\tens\hh{S})\\
&=\frac{1}{h(\eta)^2}(\cF\tens\id)\comp\alpha\comp\mu\comp\sigma\comp\sigma\comp(\hh{\cF}\tens\hh{\cF})\comp(\hh{S}\tens\hh{S}).
\end{split}
\end{equation}
Now since $\cF$ and $\hh{\cF}\comp\hh{S}$ are linear isomorphisms, comparing \eqref{to1} with \eqref{to2} we see that
\[
\hh{\alpha}\comp\hh{m}=(\hh{m}\tens{m_\sB})\comp(\id\tens\sigma\tens\id)\comp(\hh{\alpha}\tens\hh{\alpha})
\]
if and only if
\[
\alpha\comp\mu=(\mu\tens{m_\sB})(\id\tens\sigma\tens\id)\comp(\alpha\tens\alpha).
\]

Proof of \eqref{dualactionalg3}: From the fact that the unit $\hh{\I}$ of $\C(\hh{\GG})$ is $\tfrac{1}{h(\eta)}\cF(\eta)$ and Proposition \ref{alhatInne} we immediately find that
\[
\hh{\alpha}(\hh{\I})=\frac{1}{h(\eta)}(\cF\tens\id)\alpha\cF^{-1}\cF(\eta)=\frac{1}{h(\eta)}(\cF\tens\id)\alpha(\eta).
\]
This shows that $\hh{\alpha}(\hh{\I})=\hh{\I}\tens\I$ if and only if $\alpha(\eta)=\eta\tens\I$.

Proof of \eqref{dualactionalg2}: Remembering that antipodes of finite quantum groups are $*$-preserving maps we compute
\[
\hh{\alpha}\comp*=\frac{1}{h(\eta)}(\cF\tens\id_{\sB})\comp\alpha\comp\hh{\cF}\comp\hh{S}\comp*
=\frac{1}{h(\eta)}(\cF\tens\id_{\sB})\comp\alpha\comp{\bullet}\comp\hh{\cF}\comp\hh{S}
\]
and
\[
(*\tens*)\comp\hh{\alpha}=\frac{1}{h(\eta)}(*\tens*)(\cF\tens\id_{\sB})\comp\alpha\comp\hh{\cF}\comp\hh{S}
=(\cF\tens\id_{\sB})\comp({\bullet}\tens*)\comp\alpha\comp\hh{\cF}\comp\hh{S},
\]
so, as all the maps are invertible, we see that
\[
(*\tens*)\comp\hh{\alpha}=\hh{\alpha}\comp*
\]
if and only if
\[
\alpha\comp{\bullet}=({\bullet}\tens*)\comp\alpha.
\]

Proof of \eqref{dualactionalg4}: We have
\[
(\hh{h}\tens\id_{\sB})\comp\hh{\alpha}=\frac{1}{h(\eta)}\bigl((\hh{h}\comp\cF)\tens\id_{\sB}\bigr)\comp\alpha\comp\hh{\cF}\comp\hh{S}=
(\epsilon\tens\id_{\sB})\comp\alpha\comp\hh{\cF}\comp\hh{S}
\]
and
\[
\hh{h}(\cdot)\I_{\sB}=\bigl(\:\!\hh{h}\comp\hh{S}\bigr)(\cdot)\I_{\sB}=\bigl(\epsilon\comp\hh{\cF}\comp\hh{S}\:\!\bigr)(\cdot)\I_{\sB}
\]
so that
\[
(\hh{h}\tens\id_{\sB})\hh{\alpha}=\hh{h}(\cdot)\I_{\sB}
\]
if and only if
\[
(\epsilon\tens\id_{\sB})\alpha=\epsilon(\cdot)\I_{\sB}.
\]

Formula \eqref{displayed} is a consequence of Lemma \ref{Fouriter}, Equation \eqref{Delhat} and the fact that the antipode is involutive.
\end{proof}

\begin{proposition} \label{equivaut}
Let $\sB$ be a unital \cst-algebra and $\alpha\colon\C(\GG)\to\C(\GG)\tens\sB$ represent a quantum family of invertible maps on $\GG$. Then the following conditions are equivalent:
\begin{enumerate}
\item $\alpha$ preserves the convolution multiplication, preserves the convolution adjoint and the Haar element (we will also say in short that $\alpha$ \emph{preserves the convolution structure});
\item $\hh{\alpha}$ represents a quantum family of invertible maps on $\hh{\GG}$.
\end{enumerate}
Moreover in that case we have $\hh{\;\!\!\hh{\alpha}}=\alpha$.
\end{proposition}

\begin{proof}
We claim that $\alpha$ satisfies Podle\'s condition implies that $\hh{\alpha}$ satisfies Podle\'s condition. Indeed, suppose that given $a\in\C(\GG)$, there exists elements $a_1,\dotsc,a_n\in\C(\GG)$ and $q_1,\dotsc,q_n\in\sB$ such that $ \sum\limits_{i=1}^n\alpha\bigl(S(a_i)\bigr)(\I\tens{q_i})=a\tens\I$. Then
\[
\sum_{i=1}^n\hh{\alpha}\bigl(\hh{S}\cF(a_i)\bigr)(\I\tens{q_i})
=\sum_{i=1}^n(\cF\tens\id)\alpha\bigl(S(a_i)\bigr)(\I\tens{q_i})
=\cF(a)\tens\I.
\]
This proves the claim. Thus, the  equivalence is an immediate consequence of Lemma \ref{dualactionalg}.

The second statement follows from the fact that if $\alpha$ preserves both the usual adjoint and the convolution adjoint, it also preserves the antipode (cf.~\eqref{convadj}, in fact, preservation of any two of these maps implies the preservation of the third) and formula \eqref{displayed}.
\end{proof}

\begin{definition}\label{defquantaut}
Let $\sB$ be a unital \cst-algebra and let $\alpha\colon\C(\GG)\to\C(\GG)\tens\sB$ represent a quantum family of invertible maps. We say that $\alpha$ represents a \emph{quantum family of automorphisms} of $\GG$ (indexed by $\sB$) if the (equivalent) conditions from Proposition \ref{equivaut} hold.
\end{definition}

\begin{corollary}\label{corquantaut}
Let $\sB$ be a unital \cst-algebra and $\alpha\colon\C(\GG)\to\C(\GG)\tens\sB$. Then the following conditions are equivalent
\begin{enumerate}
\item\label{corquantaut1} $\alpha$ represents a quantum family of automorphisms of $\GG$;
\item\label{corquantaut2} $\hh{\alpha}$ represents a quantum family of automorphisms of $\hh{\GG}$.
\end{enumerate}
\end{corollary}

\begin{proof}
Assume that \eqref{corquantaut1} holds. By Proposition \ref{equivaut} $\hh{\;\!\!\hh{\alpha}}=\alpha$ is a unital $*$-homomorphism satisfying Podle\'s condition, so in fact $\hh{\alpha}$ represents a quantum family of automorphisms of $\hh{\GG}$.

If \eqref{corquantaut2} holds, then by the same token as in the argument in the implication \eqref{corquantaut1}$\Rightarrow$\eqref{corquantaut2} the map $\beta=\hh{\;\!\!\hh{\alpha}}$ represents a quantum family of automorphisms of $\GG$; in particular $\hh{\!\hh{\beta}}=\beta$. But, by Lemma \ref{dualactionalg}, $\,\hh{\!\hh{\beta}}=\alpha$, which completes the proof.
\end{proof}

\begin{corollary}\label{Haarpreserv}
Let $\sB$ be a unital \cst-algebra and let $\alpha\colon\C(\GG)\to\C(\GG)\tens\sB$ represent a quantum family of automorphisms of $\GG$. Then $\alpha$ preserves the counit and the Haar state. Moreover,
\begin{equation}\label{displayed2}
\hh{\alpha}=\frac{1}{h(\eta)}(\hh{S}\comp\cF\tens\id_{\sB})\comp\alpha\comp\hh{\cF}.
\end{equation}
\end{corollary}

\begin{proof}
Observe  that the equality $\alpha(\eta)=\eta\tens\I_{\sB}$ implies that $\alpha$ preserves the counit. Indeed, if $a\in\C(\GG)$ then
\[
\begin{split}
\eta\tens\epsilon(a)\I_\sB&=\alpha\bigl(\epsilon(a)\eta\bigr)=\alpha(a\eta)= \alpha(a)\alpha(\eta)=\alpha(a)(\eta\tens\I_{\sB})\\
&=\bigl(\I_{\C(\GG)}\tens((\epsilon\tens\id_{\sB})\alpha(a))\bigr)(\eta\tens\I_{\sB})=\eta\tens\bigl((\epsilon\tens\id_{\sB})\alpha(a)\bigr),
\end{split}
\]
and as $\eta\neq{0}$ we see that $\alpha$ preserves the counit. From Corollary \ref{corquantaut} we deduce that $\hh{\alpha}$ preserves the counit of $\hh{\GG}$. But then a combination of Lemma \ref{dualactionalg}\eqref{dualactionalg4} and the equality $\hh{\;\!\!\hh{\alpha}}=\alpha$ show that $\alpha$ preserves the Haar state.

Formula \eqref{displayed2} follows from the fact that $\alpha$ preserves the antipode.
\end{proof}


\begin{proposition} \label{classquantaut}
Let $\alpha\colon\C(\GG)\to\C(\GG)\tens\sB$ represent a quantum family of automorphisms of $\GG$ and assume that $\sB$ is commutative and $X$ is a compact space such that $\sB=\C(X)$. Then there is a family of Hopf $*$-algebra automorphisms $\{\psi_x\}_{x\in{X}}$ of $\C(\GG)$ such that for any $x\in{X}$, $a \in \C(\GG)$,
\[
(\id\tens\xi_x)\alpha(a)=\psi_x(a),
\]
where $\xi_x$ is the evaluation map $\sB\ni{f}\mapsto{f(x)}\in\CC$. Moreover,  for a fixed $a\in \C(\GG)$ the elements $\psi_x(a)$ depend continuously on $x$.

If, in addition, $\GG$ is a classical finite group then each $\psi_x$ is an automorphism of $\GG$.
\end{proposition}

\begin{proof}
If $\sB=\C(X)$ then for each $x\in{X}$ we define $\psi_x=(\id\tens\xi_x)\comp\alpha$. Then $\psi_x\colon\C(\GG)\to\C(\GG)$ is a unital $*$-homomorphism and it is a standard fact that the map $x\mapsto\psi_x(a)$ is continuous for any $a\in\C(\GG)$.

Moreover, $\psi_x$ is an automorphism of $\C(\GG)$. Indeed, if $\psi_x(a)=0$ for some $a\in\C(\GG)$ then $\psi_x(a^*a)=0$. Then $h(a^*a)=h\bigl(\psi_x(a^*a)\bigr)=0$ because $\alpha$ preserves the Haar measure of $\GG$, i.e.
\[
(h\tens\id)\alpha(b)=h(b)\I,\qquad(b\in\C(\GG)).
\]
As $h$ is faithful, we see that $a=0$. This shows that $\ker\psi_x=\{0\}$, and so $\psi_x$ is a linear automorphism of the finite dimensional vector space $\C(\GG)$.

The map $\psi_x$ is also a $*$-homomorphism for the convolution product and convolution adjoint on $\C(\GG)$. This shows that $\psi_x\comp{S}=S\comp\psi_x$ (cf.~formula \eqref{convadj}).

Now let us see that $\psi_x$ preserves the comultiplication. For any $a,b\in\C(\GG)$ we have
\[
a\star{b}=\psi_x^{-1}\bigl(\psi_x(a)\star\psi_x(b)\bigr).
\]
Expanding this according to the definition of the convolution product yields
\begin{equation}\label{convconv}
(h\tens\id)\Bigl(\bigl((S\tens\id)\Delta(b)\bigr)(a\tens\I)\Bigr)=
(h\tens\psi_x^{-1})\Bigl(\bigl((S\tens\id)\Delta(\psi_x(b))\bigr)\bigl(\psi_x(a)\tens\I\bigr)\Bigr).
\end{equation}
Now since $h\comp\psi_x=h$ we can substitute $h\comp\psi_x^{-1}$ for $h$ on the right hand side and rewrite \eqref{convconv} as
\[
\begin{split}
(h\tens\id)\Bigl(\bigl((S\tens\id)\Delta(b)\bigr)(a\tens\I)\Bigr)
&=(h\tens\id)(\psi_x^{-1}\tens\psi_x^{-1})\Bigl(\bigl((S\tens\id)\Delta(\psi_x(b))\bigr)\bigl(\psi_x(a)\tens\I\bigr)\Bigr)\\
&=(h\tens\id)\Bigl(\bigl((S\tens\id)(\psi_x^{-1}\tens\psi_x^{-1})\Delta(\psi_x(b))\bigr)(a\tens\I)\Bigr),
\end{split}
\]
as $\psi_x^{-1}$ commutes with $S$. We arrive at
\[
\bigl((ah\comp{S})\tens\id\bigr)\Bigl((\psi_x^{-1}\tens\psi_x^{-1})\Delta\bigl(\psi_x(b)\bigr)\Bigr)
=\bigl((ah\comp{S})\tens\id\bigr)\bigl(\Delta\bigl(b)\bigr).
\]
Taking into account the fact that $S$ is a linear automorphism of $\C(\GG)$ and faithfulness of $h$, we see that as $a$ varies over $\C(\GG)$ the functionals $ah\comp{S}$ fill the whole space $\C(\GG)^*$. This immediately implies that
\[
(\psi_x^{-1}\tens\psi_x^{-1})\comp\Delta\comp\psi_x=\Delta,
\]
so that $\psi_x$ is a Hopf algebra automorphism of $\C(\GG)$.
\end{proof}

\begin{definition}
Let $\GG$ be a finite quantum group and let $\QAUT(\GG)$ denote the category of quantum families of automorphisms of $\GG$: its objects are pairs $(\sB, \alpha)$, where $\sB$ is a unital \cst-algebra and $\alpha\colon\C(\GG)\to\C(\GG)\tens\sB$ represents a quantum family of automorphisms of $\GG$ (understood as in Definition \ref{defquantaut}) and a morphism from $(\sB, \alpha)$ to $(\sB', \alpha')$  is defined as a unital $*$-homomorphism from $\sB'$ to $\sB$ intertwining $\alpha'$ and $\alpha$ -- note the `inversion of arrows' representing the fact that we think of $\sB$ as the algebra of continuous functions on a `quantum space'.
\end{definition}

In the next lemma we will use the notion of composition of quantum families of maps introduced in \cite[Section 3]{qs}. We will only need this notion in the context of quantum families of maps $\GG\to\GG$: let $\sB$ and $\sC$ be \cst-algebras and let $\beta\colon\C(\GG)\to\C(\GG)\tens\sB$ and $\gamma\colon\C(\GG)\to\C(\GG)\tens\sC$ represent quantum families of maps $\GG\to\GG$. The \emph{composition} of $\beta$ and $\gamma$ is by definition the quantum family of maps $\GG\to\GG$ represented by
\[
(\beta\tens\id)\comp\gamma\colon\C(\GG)\to\C(\GG)\tens\sB\tens\sC=\C(\GG)\tens(\sB\tens\sC).
\]
We will denote the composition of $\beta$ and $\gamma$ by the symbol $\beta\vt\gamma$.

\begin{lemma}\label{compoAut}
Let $\sB$ and $\sC$ be \cst-algebras and let $\beta\colon\C(\GG)\to\C(\GG)\tens\sB$ and $\gamma\colon\C(\GG)\to\C(\GG)\tens\sC$ represent quantum families of automorphisms of $\GG$. Then the composition $\beta\vt\gamma\colon\C(\GG)\to\C(\GG)\tens(\sB\tens\sC)$ represents a quantum family of automorphisms of $\GG$.
\end{lemma}

\begin{proof}
The fact that a given $*$-homomorphism represents a quantum family of automorphisms of $\C(\GG)$ means that it
\begin{itemize}
\item preserves the Haar measure of $\GG$,
\item is a homomorphism with respect to the convolution product,
\item is a $*$-map for the convolution adjoint on $\C(\GG)$,
\item satisfies the Podle\'s condition.
\end{itemize}

If $\beta\colon\C(\GG)\to\C(\GG)\tens\sB$ and $\gamma\colon\C(\GG)\to\C(\GG)\tens\sC$ satisfy the first three of the above three requirements then one can show by quite trivial direct computation that so does $\beta\vt\gamma$.\footnote{
For example take $a,b\in\C(\GG)$ and let $\gamma(a)=\sum\limits_{i}a_i\tens{x_i}$, $\gamma(b)=\sum\limits_{j}b_j\tens{y_j}$ and $\beta(a_i\star{b_j})=\sum\limits_{k}c^{ij}_k\tens{z_k}$. We have
\[
(\beta\vt\gamma)(a\star{b})=(\beta\tens\id)\gamma(a\star{b})
=(\beta\tens\id)\gamma\biggl(\sum_{i,j}a_i\star{b_j}\tens{x_iy_j}\biggr)
=\sum_{i,j}\beta(a_i\star{b_j})\tens{x_iy_j}.
\]
Further, for each $i$ and $j$ let $\beta(a_i)=\sum\limits_pu^i_p\tens{w_p}$ and $\beta(b_j)=\sum\limits_qv^j_q\tens{r_q}$. As the map $\beta$ represents a quantum family of automorphisms of $\GG$, we have $\sum\limits_{k}c^{ij}_k\tens{z_k}=\sum\limits_{p,q}u^i_p\star{v^j_q}\tens{w_pr_q}$ for each $i,j$. Now denoting by $\mu$ and $m$ the convolution multiplication on $\C(\GG)$ and the multiplication on $\sB\tens\sC$ and by $\sigma$ the flip $\sB\tens\sC\tens\C(\GG)\to\C(\GG)\tens\sB\tens\sC$ we compute
\[
\begin{split}
(\mu\tens{m})(\id\tens\sigma\tens\id)&\bigl((\beta\vt\gamma)(a)\tens(\beta\vt\gamma)(b)\bigr)
=\sum\limits_{i,p,j,q}u^i_p\star{v^j_q}\tens{w_pr_q}\tens{x_iy_j}\\
&=\sum\limits_{i,j}\biggl(\sum_{p,q}u^i_p\star{v^j_q}\tens{w_pr_q}\biggr)\tens{x_iy_j}
=\sum\limits_{i,j,k}c^{ij}_k\tens{z_k}\tens{x_iy_j}=(\beta\vt\gamma)(a\star{b}).
\end{split}
\]
}
However even those simple computations can be avoided when we realize that
\begin{itemize}
\item invariance of a given state is preserved under composition of quantum families of maps (\cite[Proposition 14]{qs}),
\item $\beta\vt\gamma$ is a $*$-homomorphism for the convolution multiplication and convolution adjoint on $\C(\GG)$ by definition of the composition of quantum families.
\end{itemize}

Finally, one can show that composition of quantum families satisfying Podle\'s condition also satisfies Podle\'s condition (\cite{qinvert}). In particular $\beta\vt\gamma$ satisfies the Podle\'s condition, and so it represents a quantum family of automorphisms of $\C(\GG)$.
\end{proof}

We introduce one more piece of terminology.

\begin{definition}\sloppy
Let $\HH$ be a compact quantum group, $\GG$ a finite quantum group and $\alpha\colon\C(\GG)\to\C(\GG)\tens\C(\HH)$ be an action of $\HH$ on $\C(\GG)$: recall that this means that $\alpha$ represents a quantum family of invertible maps on $\GG$ and the \emph{action equation} holds:
\[(\id_{\C(\GG)} \tens\Delta_{\HH}) \comp\alpha  = (\alpha \tens\id_{\HH}) \comp\alpha.  \]
If in addition $\alpha$ represents a quantum family of automorphisms of $\GG$ we say that $\alpha$ is an action of $\HH$ on $\GG$ by (quantum) automorphisms.
\end{definition}

The next theorem is the key existence result of the note.

\begin{theorem} \label{existquantaut}
Let $\GG$ be a finite quantum group. The category $\QAUT(\GG)$ admits a (necessarily unique up to an isomorphism) final object, $(\sB_u, \alpha_u)$. Moreover the algebra $\sB_u$ admits a unique structure of the algebra of continuous functions on a compact quantum group (to be denoted $\qAut(\GG)$) such that $\alpha_u$ defines an action of $\qAut(\GG)$ on $\GG$. The quantum group $\qAut(\GG)$ will be called the quantum automorphism group of $\GG$; it is a quantum subgroup of the Wang's quantum automorphism group of $\bigl(\C(\GG),h\bigr)$.
\end{theorem}

\begin{proof}
Let $\KK$ be Wang's quantum automorphism group of $(\C(\GG),h)$  -- see \cite{Wang} -- and let $\bbeta\colon\C(\GG)\to\C(\GG)\tens\C(\KK)$ be its action on $\GG$. The universal property of $(\KK,\bbeta)$ says that for any unital \cst-algebra $\sB$ and any unital $*$-homomorphism $\alpha\colon\C(\GG)\to\C(\GG)\tens\sB$ preserving $h$ and satisfying the Podle\'s condition there exists a unique $\Lambda\colon\C(\KK)\to\sB$ such that $\alpha=(\id\tens\Lambda)\comp\bbeta$.

Denote by $m_{\C(\KK)}$ the multiplication map $\C(\KK)\atens\C(\KK)\to\C(\KK)$ and by $\mu$ the convolution multiplication $\C(\GG)\atens\C(\GG)\to\C(\GG)$.
Now let $\sS$ be the quotient of $\C(\KK)$ by the smallest (closed two sided) ideal containing the following three sets:
\[
\begin{split}
&\bigl\{(\omega\tens\id)\bigl((\mu\tens{m_{\C(\KK)}})(\id\tens\sigma\tens\id)(\bbeta(a)\tens\bbeta(b)-\bbeta(a\star{b})\bigr)\st{a,b}\in\C(\GG),\;\omega\in\C(\GG)^*\bigr\},\\
&\bigl\{(\omega\tens\id)\bigl((\bullet\tens*)\bbeta(a)-\bbeta(a^\bullet)\bigr)\st{a}\in\C(\GG),\;\omega\in\C(\GG)^*\bigr\},\\
&\bigl\{(\epsilon\tens\id)\bbeta(a)-\epsilon(a)\I\st{a}\in\C(\GG)\bigr\},
\end{split}
\]
where $\sigma$ denotes the flip map $\C(\KK)\tens\C(\GG)\to\C(\GG)\tens\C(\KK)$. Then let $\balpha=(\id\tens\pi)\comp\bbeta$, with $\pi$ the quotient map $\C(\KK)\to\sS$. Clearly $\balpha$ represents a quantum family of automorphisms of $\GG$. Moreover it obviously has the following universal property: if $\sC$ is a unital \cst-algebra and $\gamma\colon\C(\GG)\to\C(\GG)\tens\sC$ represents a quantum family of automorphisms of $\GG$ then there exists a unique $\Lambda\colon\sS\to\sC$ such that
\[
\gamma=(\id\tens\Lambda)\comp\balpha.
\]

Showing that there exists a unique $\Delta\colon\sS\to\sS\tens\sS$ such that
\begin{equation}\label{alal}
(\id\tens\Delta)\comp\balpha=(\balpha\tens\id)\comp\balpha
\end{equation}
is now quite standard: the right hand side of \eqref{alal} represents a quantum family of automorphisms of $\GG$. Therefore it is of the form given by the left hand side of \eqref{alal} for a unique $\Delta\colon\sS\to\sS\tens\sS$. Coassociativity of $\Delta$ follows from associativity of the operation of composition of quantum families (\cite[Proposition 5]{qs}).

In the same way as in the proofs of \cite[Proposition 12, Theorem 16(6), Theorem 21(6)]{qs} one shows that the quotient map $q\colon\C(\KK)\to\sS$ satisfies
\[
(q\tens{q})\comp\Delta_\KK=\Delta\comp{q}.
\]
Since $q$ is a continuous surjection, the density conditions
\[
\Delta(\sS)(\I\tens\sS)\quad\text{and}\quad(\sS\tens\I)\Delta(\sS)\quad\text{are dense in}\quad\sS\tens\sS
\]
hold and so we can define a compact quantum group $\qAut(\GG)$ by putting $\C\bigl(\qAut(\GG)\bigr)=\sS$. We have already shown that $\qAut(\GG)$ is the final object in the category $\QAUT(\GG)$.
\end{proof}

\begin{proposition}
The compact quantum groups $\qAut(\GG)$ and $\qAut(\hh{\GG})$ are (canonically) isomorphic.
\end{proposition}

\begin{proof}
Follows from Corollary \ref{corquantaut}.
\end{proof}

Note that in view of Theorem \ref{existquantaut} and Proposition \ref{classquantaut} it is natural to ask how one can define the quantum group of all quantum \emph{inner} automorphisms of a given quantum group. We intend to address this question in later work, here mentioning only that the notion of \emph{classical} inner automorphisms of a compact quantum group was recently introduced and studied in \cite{Issan}.

\section{Quantum automorphisms of a finite group}

Let $\Gamma$ be a finite group. Denote by $\cA$ the vector space of all complex valued functions on $\Gamma$. We will use its canonical basis $\{\delta_x\}_{x\in\Gamma}$. In what follows we shall study a linear map $\alpha\colon\cA\to\cA\tens\sB$, where $\sB$ is a unital \cst-algebra. The map $\alpha$ defines a matrix $P\in{M_{|\Gamma|}}\bigl(\sB\bigr)$ by
\begin{equation}\label{actionform}
\alpha(\delta_y)=\sum_{x\in\Gamma}\delta_x\tens{p_{x,y}}.
\end{equation}

The following three propositions follow from elementary calculations. \nopagebreak

\begin{proposition} \label{unitstar}
The map $\alpha$ is a unital $*$-homomorphism for the pointwise $*$-algebra structure on $\cA$ if and only if
\begin{subequations}\label{pRel1}
\begin{align}
p_{x,y}^*&=p_{x,y},\qquad{x},y,\in\Gamma,\label{pStar}\\
p_{x,y}^2&=p_{x,y},\qquad{x},y,\in\Gamma,\label{pSquare}\\
\sum_{y}p_{x,y}&=\I,\qquad\quad\:\!{x}\in\Gamma.\nonumber
\end{align}
\end{subequations}
\end{proposition}

\begin{proposition}
$\alpha$ preserves the convolution product on $\cA$ if and only if
\begin{equation}\label{auto}
p_{x,yz}=\sum_{u\in\Gamma}p_{u,y}p_{u^{-1}x,z},\qquad{x},y,z\in\Gamma.
\end{equation}
\end{proposition}

\begin{proposition}\label{star-map}
$\alpha$ is a $*$-map for the involution $\bullet$ on $\cA$ if and only if
\[
p_{x,y}^*=p_{x^{-1},y^{-1}},\qquad{x},y\in\Gamma.
\]
\end{proposition}

The next fact is likely well-known, but we could not locate an explicit statement in the literature (note in particular that Section 3 of \cite{Wang} lists the preservation of the counting measure by the quantum action as one of the assumptions appearing in the definition of the quantum permutation groups).

\begin{lemma}\label{countpreserved}
Let $\QG$ be a compact quantum group, let $n\in \NN$ and let $\alpha\colon\CC^n\to\CC^n\tens\C(\QG)$ be an action of $\QG$ on the algebra $\CC^n$, say given by the formulas
\[\alpha(\delta_i)=\sum_{j=1}^n\delta_j\tens{p_{j,i}},\;\; i=1,\ldots, n.\]
 Then the matrix $(p_{i,j})_{i,j=1}^n$ is a  magic unitary (i.e.~a unitary matrix, whose entries are projections); in particular projections in each row and column are mutually orthogonal.
\end{lemma}
\begin{proof}
By the computations identical to those needed in Proposition \ref{unitstar} it suffices to show that $\alpha$ must preserve the counting measure of $\CC^n$, that is
\begin{equation}\label{h-inv}
\sum_{j=1}^np_{j,i}=\I, i=1,\ldots, n.
\end{equation}
By Proposition 2.3 of \cite{Soltan}, there exists a faithful state $\rho$ on $\CC^n$ which is preserved by $\alpha$. This means that there exists a sequence $(c_i)_{i=1}^n$ of strictly positive numbers summing to $1$ such that for each $j\in \{1,\ldots,n\}$ we have $\sum_{j=1}^n c_j p_{j,i}=c_i 1$. Relabeling the elements if necessary we can assume that there exists $k\in\{1,\ldots,n\}$ such that $c_1=\cdots=c_k$ and if $l\in \{k+1,\ldots,n\}$ then $c_l> c_1$. If $k=n$ then we are done, as then $\rho$ corresponds to the normalised counting measure.
Consider then the case when $k<n$. As each $p_{i,j}$ is a projection, the equality
\[ c_1 1 = \sum_{j=1}^n c_j p_{j,1}\]
implies that $p_{l,1}=0$ if  $l>k$. Similarly $p_{l,i}=0$ for each $i\leq k$, $l>k$. This means also that
\[ 1 = \sum_{j=1}^k p_{j,i}, \;\; i\leq k.\]
But then
\[ \sum_{j=1}^k \sum_{i=1}^n p_{j,i} = k1 =\sum_{i=1}^k \sum_{j=1}^k p_{j,i} ,\]
and as we are dealing with the sums of positive operators we must actually have $p_{l,i}=0$ for $i>k$, $l\leq k$. This means that the matrix $(p_{j,i})_{i,j=1}^n$ is in fact a block-diagonal matrix which has a magic unitary as a top-left $k\times k$ block. An obvious finite induction (working in the next step with $c_{k+1}=\cdots=c_{k+l}<c_{k+l+1}$) shows that the whole matrix is a magic unitary and thus the action preserves the counting measure.
\end{proof}


The next proposition shows that if only the action of a compact quantum group $\GG$ on $\Gamma$ 
preserves the convolution product, it must be an action by automorphisms.

\begin{proposition} \label{dualact}
Assume that $\alpha$ is an action of $\GG$ on $\Gamma$ and that $\alpha$ preserves the convolution product on $\cA$. Then
\begin{align*}
p_{e,y}&=\delta_{y,e}\I,\qquad{y}\in\Gamma,\\
p_{x,e}&=\delta_{x,e}\I,\qquad{x}\in\Gamma.
\end{align*}
Further $\alpha$ is an action of $\GG$ on $\Gamma$ by quantum automorphisms.
\end{proposition}

\begin{proof}
The assumption that $\alpha$ is an action of $\GG$ on $\Gamma$ means, via Lemma \ref{countpreserved}, that the matrix $(p_{x,y})_{x,y\in \Gamma}$ is a magic unitary.

Note that once this is known, the formula \eqref{auto} implies the following equality:
\begin{equation}\label{auto2}
p_{u,y}p_{x,yz}=p_{u,y}p_{u^{-1}x,z},\qquad{u},x,y,z\in\Gamma.
\end{equation}
Inserting $z=e$ and $x=u$ in the above equality we get $p_{u,y}p_{u,y}=p_{u,y}p_{e,e}$, i.e.~$p_{u,y}=p_{u,y}p_{e,e}$. As the projections we consider are self-adjoint, we also have $p_{u,y}=p_{e,e}p_{u,y}$. Putting $u=e$ in this equation and summing over $y$, we have $p_{e,e}=\I$. Then, since $\sum\limits_yp_{e,y}=\I= \sum\limits_yp_{y,e}$, we have 
for each $x\in\Gamma$
\[
\begin{split}
p_{e,x}&=\delta_{e,x}\I,\\
p_{x,e}&=\delta_{x,e}\I.
\end{split}
\]
The formulas above imply in particular that $\alpha$ preserves the Haar element of $\cA$, i.e.~$\delta_e$.

Return now to formula \eqref{auto2} and put $x=e$, $z=y^{-1}$. This yields $p_{u,y}p_{e,e}=p_{u,y}p_{u^{-1},y^{-1}}$, so also
\[
p_{u,y}=p_{u,y}p_{u^{-1},y^{-1}},\qquad{u},y\in\Gamma.
\]
Replacing $u$ by $u^{-1}$ and $y$ by $y^{-1}$ we see that
\[
p_{u^{-1},y^{-1}}=p_{u^{-1},y^{-1}}p_{u,y}
\]
and self-adjointness yields
\[
p_{u^{-1},y^{-1}}=p_{u,y},\qquad{u},y\in \Gamma.
\]
This shows that $\alpha$ is a unital $*$-homomorphism with respect to the convolution structure on $\cA$. Together with results of Section \ref{Secdefautomor}, it ends the proof.
\end{proof}

Recall that the counits $\epsilon$ and $\hh{\epsilon}$ and Haar measures $h$ and $\hh{h}$ of $(\cA,\Delta)$ and $(\cA,\hh{\Delta})$ respectively are given by
\[
\begin{aligned}
\epsilon(\delta_x)&=\hh{h}(\delta_x)=\delta_{x,e},\\
\hh{\epsilon}(\delta_x)&=h(\delta_x)=1,
\end{aligned}
\qquad{x}\in\Gamma.
\]

\subsection{Quantum automorphisms and order}

In this subsection we show that, as in the classical case, quantum automorphisms in a natural sense preserve the order of elements and show that certain quantum automorphism groups are classical. Recall that the \emph{order} of an element $x\in\Gamma$ is defined as $\ord(x)=\min\{n\in\NN\st{x^n=e}\}$.

\begin{proposition} \label{order}
Let $\GG$ be a compact quantum group and let $\alpha\colon\cA\to\cA\tens\C(\GG)$ given by the prescription \eqref{actionform} define an action of $\GG$ by quantum automorphism. Then the following conditions are satisfied:
\begin{enumerate}
\item\label{statement-i} $p_{x,y}=0$ if $x,y\in\Gamma$, $\ord(x)\neq\ord(y)$;
\item $p_{x,y}p_{x^n,z}=p_{x^n,z}p_{x,y}$, for $x,y,z\in\Gamma$, $n\in\NN$;
\item $p_{y,x}p_{z,x^n}=p_{z,x^n}p_{y,x}$, for $x,y,z\in\Gamma$, $n\in\NN$;
\item $p_{x^n,y^n}\geq{p_{x,y}}$ for $x,y\in\Gamma$, $n\in\NN$.
\end{enumerate}
\end{proposition}

\begin{proof}
We begin by recalling the formula \eqref{auto2} satisfied by the elements $\{p_{x,y}\st{x},y\in\Gamma\}$ determining the action $\alpha$. Note first that this can be rewritten as
\begin{equation}\label{auto3}
p_{x,y}p_{z,u}=p_{x,y}p_{xz,yu},\qquad{u},x,y,z\in\Gamma.
\end{equation}
Note that the last expression is symmetric with respect to the swapping of rows and columns of the matrix $(p_{x,y})_{x,y\in\Gamma}$.

Apply it with $z=x$. This gives $p_{x,y}p_{x,u}=p_{x,y}p_{x^2,yu}$, so we obtain
\[
p_{x,y}p_{x^2,yu}=\delta_{yu}p_{x,y},\qquad{x},y,u\in\Gamma.
\]
The expression on the left can be further rewritten using \eqref{auto3}, so that we get
\[
p_{x,y}p_{x^3,y^2u}=\delta_{yu}p_{x,y},\qquad{x},y,u \in \Gamma,
\]
and further inductively
\begin{equation} \label{inductively}
p_{x,y}p_{x^{k+1},y^k u}=\delta_{yu}p_{x,y},\qquad{x},y,u\in\Gamma,\;k\in\NN.
\end{equation}
Note that by selfadjointness the same holds with the projections on the left switching sides (in particular, the two projections featuring on the left commute).

If then say $\ord(x)=l<\ord(y)$ we obtain (putting $y=u$)
\[
p_{x,y}=p_{x,y}p_{x^l,y^l}=p_{x,y}p_{e,y^l}=0.
\]
This proves statement \eqref{statement-i} in the proposition. It is not difficult to see that all the other statements are direct consequences of the remarks and formulas in the above proof.
\end{proof}

The following theorem and its proof were communicated to us by the anonymous referee (the original version contained the result valid only for cyclic groups of prime order).

\begin{theorem} \label{cyclic}
Let $\Gamma$ be a finite cyclic group. Then the quantum automorphism group of $\Gamma$ is classical.
\end{theorem}
\begin{proof}
We will use the notation of the last proposition and write $n$ for the order of $\Gamma$. It suffices to show that if $x\in \Gamma$ is a generator, $y, z, t \in \Gamma$ and $d \in \NN$ then $p_{x^d,y}$ commutes with $p_{z,t}$. To establish the latter fact one can assume that $d$ divides $n$ (otherwise one may change the generator with which we start) and so we do. Note also that if $d=1$ then the corresponding commutation follows from (2) in Proposition \ref{order} (this in fact ends the proof if $\Gamma$ is of a prime order).

Observe further that if $y$ is not of order $n/d$, then by (1) in Proposition \ref{order} we must have $p_{x,y}=0$, so we can further assume that $y=s^d$ for some $s \in \Gamma$. Consider then the element $p_{x^d, s^d}$. If $v \in \Gamma$  then by \eqref{inductively} applied to $k=d-1$, $y=v$ and $u=v^{-(d-1)}s^d$ we see that  $p_{x,v} p_{x^d, s^d}$ equals $p_{x,v}$ if $v^d=s^d$ and 0 otherwise. Summing the latter equalities over $v$ yields
\[ p_{x^d, s^d} = \sum_{v\in \Gamma, v^d=s^d}\, p_{x,v}.\]
The operators on the right hand side of the above commute with any $p_{z,t}$ by the first paragraph, so the proof is finished.

\end{proof}


\subsection{Quantum automorphisms of a dual of a finite group}

Consider now another, in a sense converse, approach to the problem  studied earlier in this section. Let $\Gamma$ be again a finite group and assume that a compact quantum group $\GG$ acts on the dual of $\Gamma$. Once again, we want to identify the weakest conditions on the action, so that in fact it induces the action of $\GG$ on $\Gamma$ itself, i.e.~to provide a dual counterpart of Proposition \ref{dualact}.
The situation turns out to be equally satisfactory as in the commutative case -- the preservation of the convolution product by a given action on $\hh{\Gamma}$ already implies this is an action by automorphisms.

\begin{theorem} \label{dualactdual}
Suppose that $\Gamma$ is a finite group, $\GG$ is a compact quantum group and let $\alpha\colon\CC[\Gamma]\to\CC[\Gamma]\tens\C(\GG)$ be an action of $\GG$ on $\hh{\Gamma}$ which is a homomorphism for the convolution product of $\CC[\Gamma]$. Then $\alpha$ is an action on $\hh{\Gamma}$ by quantum automorphisms.
\end{theorem}

\begin{proof}
Let $\alpha$ as above be given by the formula
\[
\alpha(x)=\sum_{y\in\Gamma}y\tens{u_{y,x}},\qquad{x}\in\Gamma.
\]
By the well-known fact, going back to the thesis of P.\,Podle\'s (see \cite{qsIJM} for more references), concerning the decomposition of actions of compact quantum groups into isotypical components, it follows that the elements $u_{y,x}$ belong to $\Pol(\GG)$, the canonical dense Hopf $*$-subalgebra of $\C(\GG)$. Further the application of the counit $\epsilon$ of $\Pol(\GG)$ yields
\begin{equation}\label{counit}
\epsilon(u_{y,x})=\delta_{x,y},\qquad{x},y\in \Gamma.
\end{equation}
Further by the action equation we also have
\begin{equation}\label{coprod}
\Delta_{\GG}(u_{x,y})=\sum_{z\in\Gamma}u_{x,z}\tens{u_{z,y}},\qquad{x},y\in\Gamma.
\end{equation}
The fact that $\alpha$ is a homomorphism with respect to the convolution product implies that each $u_{x,y}$ is idempotent. Further, an easy norm argument implies that these elements are also contractive, and therefore self-adjoint, which means that $\alpha$ preserves the convolution adjoint. Consider then the fact that each $x\in\Gamma$ is unitary (viewed as an element of $\CC[\Gamma]$). This implies that
\[
\I\tens\I=\alpha(x^*x)=\alpha(x)^*\alpha(x)=\sum_{y,z\in\Gamma}y^*z\tens{u_{y,x}^*}u_{z,x}=\sum_{y,z\in\Gamma}y^*z\tens{u_{y,x}}u_{z,x}.
\]
In particular
\[
\sum_{y\in\Gamma}u_{y,x}=\sum_{y\in\Gamma}u_{y,x}u_{y,x}=\I.
\]
Further, as for $x, z \in \Gamma$, $x \neq z$, we have $x \star z = 0$, it follows that $u_{y,x} u_{y,z} = 0$ for each $y\in \Gamma$. Thus the matrix $(u_{x,y})_{x, y \in \Gamma}$ is a matrix of self-adjoint projections mutually orthogonal in each row and column and such that the sum of entries in each column is equal to $1$. It remains to see that the corresponding sum in each row is equal to one, as then it will follow that $\alpha$ preserves the convolution unit.  To this end introduce a new map, $\beta\colon\C(\Gamma)\to\C(\Gamma)\tens\C(\GG^{\opp})$, given by the formula:
\[
\beta(\delta_x)=\sum_{y\in\Gamma}\delta_y\tens{u_{x,y}},\qquad{x}\in\Gamma,
\]
where $\C(\GG^{\opp})$ is as a \cst-algebra isomorphic to $\C(\GG)$, but has a `tensor flipped' coproduct.
It is easy to check that $\beta$ is a unital $*$-homomorphism and it satisfies the action equation. We only verify the latter (choosing first $x \in \Gamma$):
\[
\begin{split}
(\id_{\C(\Gamma)}\tens\Delta_{\GG^{\opp}})\beta(\delta_x)&=\sum_{y\in\Gamma}\delta_y\tens\Delta_{\GG^{\opp}}(u_{x,y})
=\sum_{y\in\Gamma}\delta_y\tens\biggl(\sum_{z\in\Gamma}u_{z,y}\tens{u_{x,z}}\biggr)\\
&=\sum_{z\in\Gamma}\beta(\delta_z)\tens{u_{x,z}}=(\beta\tens\id_{\GG^{\opp}})\beta(\delta_x),
\end{split}
\]
where in the second equality we used \eqref{coprod}. It remains to notice that the identity map identifies $\Pol(\GG)$ with $\Pol(\GG^{\opp})$ and that the counit of $\Pol(\GG^{\opp})$ coincides with that of $\Pol(\GG)$. Thus \eqref{counit} implies that we have
\[ ( \id_{\C(\Gamma)} \tens \epsilon) \comp \beta = \id_{\C(\Gamma)} .\]
By Remark 2.3 of \cite{qsIJM} it follows that $\beta$ is an action of $\GG^{\opp}$ on $\C(\Gamma)$. Thus, by \cite{Wang} (we used this argument already in \eqref{h-inv}), it must preserve the uniform measure on $\Gamma$ -- and this means that  $\sum_{x\in\Gamma}u_{y,x}=\I$ and the proof is finished.
\end{proof}

\end{document}